\documentclass[12pt]{amsart}

\usepackage{amsmath}
\usepackage{amssymb}
\usepackage{amsthm}
\usepackage{hyperref}

\theoremstyle{plain}
\newtheorem{thm}{Theorem}[section]
\newtheorem{prop}[thm]{Proposition}
\newtheorem{lem}[thm]{Lemma}
\newtheorem{cor}[thm]{Corollary}
\newtheorem{fact}[thm]{Fact}

\newtheorem{maintheorem}{Theorem}

\theoremstyle{definition}
\newtheorem{defn}[thm]{Definition}
\newtheorem{rem}[thm]{Remark}

\newcommand{\sa}{\mathrm{sa}}
\newcommand{\cstar}{$\mathrm{C}^*$}
\newcommand{\wstar}{$\mathrm{W}^*$}
\newcommand{\AWstar}{$\mathrm{AW}^*$}
\newcommand{\Proj}{\operatorname{Proj}}
\newcommand{\supp}{\operatorname{supp}}
\newcommand{\fin}{\Subset}

\begin{document}

\title[Transfinite Christensen-Pedersen]{Revisiting the Transfinite Christensen-Pedersen Argument}

\author{Jananan Arulseelan}
\address{Department of Mathematics, Iowa State University, 396 Carver Hall, 411 Morrill Road, Ames, IA 50011, USA}
\email{jananan@iastate.edu}
\urladdr{https://sites.google.com/view/jananan-arulseelan}


\begin{abstract}
Christensen and Pedersen proved that every properly infinite \AWstar-algebra is monotone sequentially complete, and Sait\^{o} and Wright developed a transfinite form of their dilation argument. We revisit the transfinite construction using normality of \AWstar-algebras. Normality simplifies the limit stages by turning suprema into compressions of joins, so the construction only needs a supply of fresh orthogonal projections large enough to contain the supports of the summands at successor stages. We use this simplified proof to show that a $*$-homomorphism between \AWstar-algebras that preserves only the joins needed to encode such a sum preserves the sum itself.  We also use it to deduce order-continuity facts about $\kappa$-join-preserving $*$-homomorphisms.  We also show that a finite \AWstar-algebra has suprema for all bounded positive families whose supports have bounded total center-valued dimension.
\end{abstract}

\maketitle

\section{Introduction}

Christensen and Pedersen proved that every properly infinite \AWstar-algebra is monotone sequentially complete~\cite{CP84}.  Their proof dilates a bounded positive sequence to pairwise orthogonal projections, and then uses addability of partial isometries together with a perturbation argument to recover the least upper bound.  Sait\^{o} and Wright developed a transfinite version of this method~\cite[Sections 2 and 3]{SW91}: an \AWstar-algebra with an $\Omega$-indexed system of matrix units $(e_{ij})$ satisfying $\sum_i e_{ii} = 1$ and $e_{00} \sim 1$ has suprema for all $\Omega$-indexed positive families with bounded finite subsums~\cite[Theorem 2.5]{SW91}.  Their proof alternates between transfinite dilation and completeness at smaller ordinals, and it is the basis of their theorem that every \AWstar-factor is normal~\cite[Corollary 4.7]{SW91}.

We revisit the construction from the opposite direction.  An \AWstar-algebra is \textbf{normal} if the join of every upward directed family of projections is its supremum in the self-adjoint order.  Every \AWstar-algebra is normal by~\cite[Theorem A]{AH26}.  In a normal \AWstar-algebra, if $(q_i)_{i \in I}$ is an orthogonal family of projections with join $q$ and $E$ is a projection, then
\[
E q E = \sup_{F \fin I}\ \sum_{i \in F}E q_i E .
\]
This identity is used at every limit stage of a transfinite dilation and again at its end.  Once it is available, what remains is the question of where to place the next dilating projection.  The rotation step of~\cite{CP84}, together with the support calculation below, applied to $x_i \in (EME)_+$ produces a positive contraction $z_i$ with $E z_i E = x_i$ and $\supp(z_i)\sim\supp(x_i)$, so the $i$-th stage only needs fresh orthogonal space large enough to contain a copy of $\supp(x_i)$.

\begin{defn}\label{defn:supply}
    Let $M$ be an \AWstar-algebra, let $E \in \Proj(M)$, and let $(x_i)_{i \in I}$ be positive elements of $EME$.  A \textbf{projection-supply} for $(x_i)$ is a family $(R_i)_{i \in I}$ of pairwise orthogonal projections in $M$, each orthogonal to $E$, with $\supp(x_i) \precsim R_i$ for every $i$.
\end{defn}

\begin{maintheorem}[Projection-supply dilation]\label{mainthm:A}
    Let $M$ be a normal \AWstar-algebra, let $E \in \Proj(M)$, and let $(x_i)_{i \in I}$ be positive elements of $EME$ with $\sum_{i \in F}x_i \leq \tfrac{1}{2} E$ for every finite $F \subseteq I$.  If $(x_i)$ has a projection-supply, then there are pairwise orthogonal projections $(q_i)_{i \in I}$ in $M$ with
    \[
        E q_i E = x_i,\qquad q_i \sim \supp(x_i) \qquad (i \in I),
    \]
    and with $q = \bigvee_i q_i$, the element $E q E$ is the supremum of the finite subsums of $(x_i)$ in $M_{\sa}$.
\end{maintheorem}

When each $R_i$ contains a copy of $E$, the existence of the supremum also follows from~\cite[Theorem 2.5]{SW91} applied to a homogeneous corner.  The projection-supply condition is weaker when the supports of the $x_i$ are small compared with $E$.  This is the situation in a finite \AWstar-algebra, where no nonzero projection has infinitely many orthogonal copies.  There, the center-valued dimension can place small supports into a projection-supply.

Theorem~\ref{mainthm:A} does more than produce a supremum. It represents the supremum as the compression of a join of orthogonal projections, thus encoding it as projection-lattice data.  A $*$-homomorphism $\pi: M \to N$ between \AWstar-algebras is \textbf{$\kappa$-join-preserving} if $\pi(\bigvee_i p_i) = \bigvee_i \pi(p_i)$ for every orthogonal family of at most $\kappa$ projections in $M$.  Homomorphisms preserving joins of all orthogonal families are classical in the theory of \AWstar-algebras~\cite{Wid56}; see~\cite{HL} for a recent categorical treatment.  Our main application shows that only the joins needed to encode a given order sum need be preserved.

\begin{maintheorem}\label{mainthm:B}
    Let $M$ and $N$ be \AWstar-algebras, let $E \in \Proj(M)$, and let $(x_i)_{i \in I}$ be positive elements of $EME$ with uniformly norm-bounded finite subsums.  If $(x_i)$ has a projection-supply and $\pi: M \to N$ is an $|I|$-join-preserving $*$-homomorphism, then
    \[
        \pi\Bigl(\sup_{F \fin I}\sum_{i \in F}x_i\Bigr) = \sup_{N,\;F \fin I}\ \sum_{i \in F}\pi(x_i).
    \]
\end{maintheorem}

Section~\ref{sec:homomorphisms} derives cardinal criteria for order continuity.  Let $\pi$ be a $*$-homomorphism from a properly infinite \AWstar-algebra into an \AWstar-algebra.  If $\pi$ preserves joins of countable orthogonal families, then it preserves suprema of bounded increasing sequences (Corollary~\ref{cor:homogeneous}).  
In the von Neumann setting, $*$-homomorphisms preserving joins of orthogonal sequences of projections were studied by Bunce and Hamhalter~\cite{BuH95}, and Str\u{a}til\u{a} and Zsid\'{o} show that a countably additive positive map from a monotone $\sigma$-complete \cstar-algebra into a \wstar-algebra is sequentially normal~\cite[9.18]{SZ}.  For properly infinite domains, Corollary~\ref{cor:homogeneous} extends this to codomains that are arbitrary \AWstar-algebras, which need not admit any normal functional.  More generally, if $M$ is $\kappa$-homogeneous, every $\kappa$-join-preserving $*$-homomorphism preserves suprema of bounded increasing nets of cardinality at most $\kappa$, and if $M$ also admits $\kappa$ separating ordinary states, it actually preserves all suprema (Corollary~\ref{cor:normalhom}).  Corollary~\ref{cor:subalgebra} gives criteria under which an \AWstar-subalgebra is monotone closed in an arbitrary \AWstar-algebra; cf. Hamana's study of \AWstar-subalgebras of type $\mathrm{I}$ \AWstar-algebras~\cite{Ham83}.  For finite domains, join-preserving $*$-homomorphisms into arbitrary \AWstar-algebras preserve the suprema of families whose support dimensions fit inside the complement of their joint support (Corollary~\ref{cor:finitesource}).  Here the codomain need not be finite  To the best of our knowledge, the cardinal-local order-continuity statements of Section~\ref{sec:homomorphisms} do not appear in the literature.

Theorem~\ref{mainthm:A} also gives an existence result over arbitrary centers.

\begin{maintheorem}\label{mainthm:C}
    Let $M$ be a finite \AWstar-algebra with normalized center-valued dimension $\Delta_M$, and let $(x_i)_{i \in I}\subseteq M_+$ have uniformly norm-bounded finite subsums.  If
    \[
        \sup_{F \fin I}\ \sum_{i \in F}\Delta_M(\supp x_i) < \infty,
    \]
    then the finite subsums of $(x_i)$ have a supremum in $M_{\sa}$.
\end{maintheorem}

Finite factors also admit a metric proof; see Haagerup~\cite[Prop.~3.10]{Haa14}.  We give a projection-geometric proof that works over arbitrary centers.

Section~\ref{sec:cardinal} collects the cardinal completeness results used in Section~\ref{sec:homomorphisms}.  These include the cardinal form of the passage from summable families to monotone nets due to Sait\^{o} and Wright, a state-selection principle reducing an arbitrary bounded increasing net to a directed subfamily of cardinality at most $\kappa$ with the same supremum, and a criterion for homogeneity of \AWstar-factors with a faithful-state corner.

Theorem~\ref{mainthm:A} is applied in Section~\ref{sec:finite} only to finite \AWstar-algebras. These are normal by~\cite{Wri80}, so Theorem~\ref{mainthm:C} and its consequences do not depend on~\cite{SW91} or~\cite{AH26}.  Theorem~\ref{mainthm:B} uses normality of the domain and codomain, for which we invoke~\cite[Theorem A]{AH26}, which is proved using the normality of \AWstar-factors~\cite[Corollary 4.7]{SW91}.

\section{Preliminaries}\label{sec:prelim}

Throughout, suprema are taken in the self-adjoint order unless a corner is explicitly indicated.  For a projection $e$ in a unital \cstar-algebra $A$, the corner $eAe$ is regarded as a unital \cstar-algebra with unit $e$.  If $M$ is an \AWstar-algebra and $e \in \Proj(M)$, then $eMe$ is an \AWstar-algebra, and joins of subprojections of $e$ computed in $eMe$ agree with those computed in $M$~\cite{Ber72}.  For $x \in M$ we write $\supp(x)$ for the right support of $x$; for positive $x$ it is the least projection $p$ with $xp = x$.  An \textbf{ordinary state} is a state in the usual sense; no state in this paper is assumed to be normal.

\begin{lem}\label{lem:sumsup}
    Let $(a_i)$ and $(b_i)$ be increasing nets in $A_{\sa}$ indexed by the same directed set.  If $a = \sup_i a_i$ and $b = \sup_i b_i$ exist, then $\sup_i(a_i + b_i) = a + b$.
\end{lem}

\begin{proof}
    The element $a+b$ is an upper bound.  If $c$ is another upper bound and $\eta$ is fixed, then $a_\xi+b_\eta \leq  a_\xi+b_\xi \leq  c$ for $\xi \geq \eta$.  Taking the supremum over this cofinal tail gives $a+b_\eta \leq c$, and taking the supremum over $\eta$ gives $a+b \leq c$.
\end{proof}

Taking one net constant gives
\begin{equation}\label{eq:translate}
    \sup_i(a_i + c) = \Bigl(\sup_i a_i \Bigr) + c .
\end{equation}

\begin{lem}\label{lem:compression}
    Let $e \in A$ be a projection and let $(a_i)$ be an increasing net in $A_{\sa}$ with $a = \sup_i a_i$.  Then $eae$ is the supremum of $(e a_i e)$ in $eAe$.
\end{lem}

\begin{proof}
    Put $u = 2e-1$ and $\Phi(x) = \tfrac{1}{2}(x+uxu^*) = exe + (1-e)x(1-e)$.  Since $\operatorname{Ad}u$ is an order automorphism, Lemma~\ref{lem:sumsup} gives $\sup_i \Phi(a_i) = \Phi(a)$.  If $b \in (eAe)_{\sa}$ dominates every $e a_i e$, then $c = b + (1-e)a(1-e)$ dominates every $\Phi(a_i)$, hence $c \geq \Phi(a)$, and comparing $e$-corners gives $b \geq eae$.
\end{proof}

\begin{lem}\label{lem:corner}
    Let $e \in A$ be a projection and let $S \subseteq (eAe)_{\sa}$ have a supremum $s$ in $eAe$.  Then $s$ is the supremum of $S$ in $A_{\sa}$.
\end{lem}

\begin{proof}
    Clearly $s$ is an upper bound.  Let $b \in A_{\sa}$ dominate $S$, put $f = 1-e$, and write $b_{11} = ebe$, $b_{12} = ebf = b_{21}^*$ and $b_{22} = fbf$.  Since $f(b-y)f = b_{22}$ for $y \in S$, we have $b_{22} \geq 0$.  For $\lambda > 0$ let $c_\lambda = b_{22} + \lambda f$, which is invertible in $fAf$, and put $n = b_{12} c_\lambda^{-1} \in eAf$.  Then $n^2 = 0$, so $w = 1-n$ is invertible, and for every $y \in (eAe)_{\sa}$
    \[
        w(b-y+\lambda f)w^* = \bigl(d_\lambda-y\bigr) + c_\lambda,\qquad d_\lambda = b_{11}-b_{12} c_\lambda^{-1}b_{21} \in (eAe)_{\sa} .
    \]
    For $y \in S$ the left side is positive, so $y \leq d_\lambda$.  Hence $s \leq d_\lambda$, the right side is positive for $y = s$, and $b - s + \lambda f \geq 0$.  Letting $\lambda \to 0$ gives $b \geq s$.
\end{proof}

\begin{lem}\label{lem:annihilator}
    Let $(e_i)_{i \in I}$ be projections in an \AWstar-algebra $M$ and put $e =  \bigvee_i e_i$.  If $z \in M$ satisfies $ze_i = 0$ for every $i$, then $ze = 0$.
\end{lem}

\begin{proof}
    The right annihilator of $\{z\}$ is $rM$ for a projection $r$.  Each $e_i \leq r$, so $e \leq r$ and $ze = 0$.
\end{proof}

\begin{fact}[{\cite[Section 20, Theorem 1]{Ber72}}]\label{fact:addability}
    Let $(e_i)_{i \in I}$ and $(f_i)_{i \in I}$ be orthogonal families of projections in an \AWstar-algebra $M$ with $e_i \sim f_i$ for every $i$.  Then $\bigvee_i e_i \sim \bigvee_i f_i$.
\end{fact}

Proposition~\ref{prop:joinsup} is the case $x = E$ of~\cite[Lemma~3]{Ber83}, which Berberian attributes to Sait\^{o}~\cite{Sai81}: in a normal \AWstar-algebra, if $f_\alpha \uparrow f$ in the projection lattice, then $x f_\alpha x^* \uparrow xfx^*$ in the self-adjoint part for every $x$.  We include the short proof via Lemma~\ref{lem:compression}.

\begin{prop}\label{prop:joinsup}
    Let $M$ be a normal \AWstar-algebra, let $E \in \Proj(M)$, and let $(q_i)_{i \in I}$ be pairwise orthogonal projections with join $q$.  Then
    \[
        E q E = \sup_{EME,\;F \fin I}\ \sum_{i \in F} E q_i E,
    \]
    and this is also the supremum in $M_{\sa}$.  In particular, if $\sum_{i \in F}E q_i E \leq \rho E$ for every finite $F$, then $E q E \leq \rho E$.
\end{prop}

\begin{proof}
    The projections $q_F = \sum_{i \in F}q_i$ increase with join $q$, so normality gives $q = \sup_F q_F$ in $M_{\sa}$.  The first identity now follows from Lemma~\ref{lem:compression}, the second assertion from Lemma~\ref{lem:corner}, and the last from leastness.
\end{proof}

\section{Projection-supply dilation}\label{sec:supply}

We recall the rotation calculation of~\cite{CP84}, keeping track of the support control needed below.

\begin{lem}\label{lem:rotation}
    Let $M$ be an \AWstar-algebra, let $E,p \in \Proj(M)$, and let $x \in (EME)_+$.  Put $y = EpE$ and suppose $\lVert y \rVert < 1$ and $x+y \leq E$.  Set
    \[
    a = (1-E)pE(1-y)^{-1},\qquad c = (1-a)x^{1/2},\qquad z = cc^* .
    \]
    Then $z$ is a positive contraction with $zp = pz = 0$, $EzE = x$ and $\supp(z) \sim \supp(x)$.  If moreover $E \leq P$ and $p \leq P$ for a projection $P$, then $\supp(z) \leq P$.
\end{lem}

\begin{proof}
    Since $\lVert y \rVert < 1$, the element $E-y$ is invertible in $EME$, and $1-y$ is invertible in $M$ with inverse $(E-y)^{-1}+(1-E)$.  We have $a = aE$ and $Ea = 0$, and $p(1-E)pE = pE-pEpE = pE(1-y)$, so $pa = pE$.  Since $x^{1/2} = Ex^{1/2}$, this gives $pc = 0$, hence $pz = zp = 0$, while $Ea = 0 = a^*E$ gives $EzE = x$.  Next, $a^*a = (1-y)^{-1}Ep(1-E)pE(1-y)^{-1} = y(1-y)^{-1}$, and the cross terms vanish after multiplication by $x^{1/2}$, so $c^*c = x^{1/2}(E-y)^{-1}x^{1/2}$.  With $T = (E-y)^{-1/2}x^{1/2}$ we have $c^*c = T^*T$ and $TT^* = (E-y)^{-1/2}x(E-y)^{-1/2} \leq E$ because $x+y \leq E$, so $\lVert c \rVert \leq 1$.  Since $a = aE$ and $Ea = 0$, we have $a^2 = 0$, so $1-a$ is invertible with inverse $1+a$.  Hence $c = (1-a)x^{1/2}$ and $x^{1/2}$ have the same right annihilator, and therefore the right support of $c$ is exactly $\supp(x)$.  The left support of $c$ is $\supp(z)$, and left and right supports in an \AWstar-algebra are equivalent, so $\supp(z)\sim\supp(x)$.  Finally, if $E,p \leq P$ then $(1-E)pE = (P-E)pE$, so the left support of $a$ is dominated by $P-E$, and the left support of $c$ is dominated by $E \vee(P-E) = P$.
\end{proof}

\begin{proof}[Proof of Theorem~\ref{mainthm:A}]
    Let $(R_i)$ be a projection-supply for $(x_i)$.  Well-order $I$ and identify it with an ordinal.  For $i \in I$ put $P_i = E\vee\bigvee_{j < i}R_j$.  We construct the $q_i$ by transfinite recursion, maintaining that they are pairwise orthogonal, that $Eq_jE = x_j$ and $q_j\sim\supp(x_j)$, and that $q_j \leq E \vee \bigvee_{k \leq j}R_k$.
    
    Suppose $q_j$ has been constructed for $j < i$, and put $p^{(i)} = \bigvee_{j < i} q_j$, so that $p^{(i)} \leq P_i$.  Proposition~\ref{prop:joinsup}, applied to $(q_j)_{j < i}$, gives $E p^{(i)} E \leq  \tfrac{1}{2}E$, and $x_i \leq  \tfrac{1}{2}E$.  Hence $\lVert Ep^{(i)}E \rVert < 1$ and $x_i + Ep^{(i)}E \leq E$.  Lemma~\ref{lem:rotation}, with $p = p^{(i)}$ and $P = P_i$, gives a positive contraction $z_i$ with $z_i p^{(i)} = p^{(i)}z_i = 0$ and $E z_i E = x_i$, whose support $s_i$ satisfies $s_i \leq P_i$ and $s_i \sim \supp(x_i)\precsim R_i$.  Choose $f_i \leq R_i$ with $f_i \sim s_i$, and a partial isometry $w_i$ with $w_i^*w_i = s_i$ and $w_i w_i^* = f_i$.  Since $R_i$ is orthogonal to $E$ and to every $R_j$ with $j < i$, it is orthogonal to $P_i$ by Lemma~\ref{lem:annihilator}, so $f_i$ is orthogonal to $E$, $s_i$ and $p^{(i)}$.  Put $t_i = (z_i - z_i^2)^{1/2}$ and
    \[
        q_i = z_i + t_i w_i^* + w_i t_i + w_i(s_i - z_i)w_i^* .
    \]
    Under the identification of $(s_i + f_i)M(s_i + f_i)$ with a $2\times2$ matrix corner induced by $w_i$, this is the Halmos projection
    \[
        \begin{pmatrix}z_i&t_i\\t_i&s_i - z_i\end{pmatrix}
    \]
    of~\cite{Hal50} (see also~\cite{BS10}).  The element $V_i = z_i^{1/2} + w_i(s_i-z_i)^{1/2}$ satisfies $V_i^*V_i = s_i$ and $V_i V_i^* = q_i$, so $q_i \sim s_i \sim \supp(x_i)$.  Since $z_i p^{(i)} = 0$, we have $s_i\perp p^{(i)}$; together with $f_i\perp p^{(i)}$ and $q_i \leq s_i + f_i$ this gives $q_i\perp q_j$ for $j < i$, and $q_i \leq s_i + f_i \leq E \vee \bigvee_{k \leq i}R_k$.  Finally $Ew_i = 0$, so $E q_i E = E z_i E = x_i$.
    
    With $q = \bigvee_i q_i$, Proposition~\ref{prop:joinsup} identifies $EqE$ as the supremum of the finite subsums of $(x_i)$ in $M_{\sa}$.
\end{proof}

\begin{defn}\label{defn:copysupply}
    Let $\kappa$ be an infinite cardinal.  A nonzero projection $E \in \Proj(M)$ has a \textbf{$\kappa$-copy supply} if there are pairwise orthogonal projections $(R_\xi)_{\xi < \kappa}$, each orthogonal to $E$ and equivalent to $E$.
\end{defn}

A $\kappa$-copy supply for $E$ is a projection-supply for every family of at most $\kappa$ positive elements of $EME$.  Rescaling in Theorem~\ref{mainthm:A} therefore gives the following.

\begin{cor}\label{cor:copysupply}
    Let $M$ be a normal \AWstar-algebra and let $E \in \Proj(M)$ have a $\kappa$-copy supply.  Then every family of at most $\kappa$ positive elements of $EME$ with uniformly norm-bounded finite subsums has a supremum of its finite subsums, and this supremum lies in $EME$.
\end{cor}

\section{Finite \AWstar-algebras}\label{sec:finite}

Let $M$ be a finite \AWstar-algebra with center $Z(M)$, and let $\Delta_M:\Proj(M)\to Z(M)_+$ be its center-valued dimension function, normalized by $\Delta_M(1) = 1$.  We use the following facts from the dimension theory of finite \AWstar-algebras~\cite[Chapter 6]{Ber72}: $p\sim q$ if and only if $\Delta_M(p) = \Delta_M(q)$; $p\precsim q$ if and only if $\Delta_M(p) \leq \Delta_M(q)$; $\Delta_M$ is additive on orthogonal projections; and for every orthogonal family $(p_i)_{i \in I}$,
\begin{equation}\label{eq:completeadditivity}
\Delta_M\Bigl(\bigvee_i p_i\Bigr) = \sup_{F \fin I}\ \sum_{i \in  F}\Delta_M(p_i)\quad\text{in }Z(M)_{\sa} .
\end{equation}
The center $Z(M)$ is a commutative \AWstar-algebra, hence monotone complete.  We also use that matrix algebras over an \AWstar-algebra are \AWstar-algebras~\cite{Ber58}, and that finiteness and the center-valued dimension pass to finite matrix amplifications.

\begin{lem}\label{lem:subadditive}
    Let $(p_i)_{i \in I}$ be projections in $M$, and suppose the sums $\sum_{i \in F}\Delta_M(p_i)$ over finite $F\subseteq I$ are bounded.  Then
    \[
        \Delta_M\Bigl(\bigvee_i p_i\Bigr) \leq \sup_{F \fin I}\ \sum_{i \in F}\Delta_M(p_i).
    \]
\end{lem}

\begin{proof}
    Well-order $I$ and put $d_i = \bigvee_{j \leq i}p_j-\bigvee_{j < i}p_j$.  The $d_i$ are pairwise orthogonal, and induction on $i$ shows $\bigvee_{j \leq i}d_j = \bigvee_{j \leq i}p_j$, so their join is $\bigvee_i p_i$.  By the parallelogram law~\cite[Theorem 5.4]{Kap51}, $(P\vee p)-P\sim p-(P\wedge p)$ for projections $P$ and $p$, so $d_i \precsim p_i$ and $\Delta_M(d_i) \leq \Delta_M(p_i)$.  Now apply~\eqref{eq:completeadditivity} to $(d_i)$.
\end{proof}

The following packing lemma is a standard application of the dimension theory; we include the short proof.

\begin{lem}\label{lem:packing}
    Let $R \in \Proj(M)$ and let $(p_i)_{i \in I}$ be projections with $\sum_{i \in F}\Delta_M(p_i) \leq \Delta_M(R)$ for every finite $F\subseteq I$.  Then there are pairwise orthogonal projections $(R_i)_{i \in I}$ with $R_i \leq R$ and $R_i \sim p_i$.
\end{lem}

\begin{proof}
    Well-order $I$ and suppose $R_j$ has been chosen for $j < i$.  Put $R_{< i} = \bigvee_{j < i}R_j$.  By~\eqref{eq:completeadditivity}, $\Delta_M(R_{<i}) = \sup_{F \fin i}\sum_{j \in F}\Delta_M(p_j)$, and since $\sum_{j \in F}\Delta_M(p_j)+\Delta_M(p_i) \leq \Delta_M(R)$ for every finite $F\subseteq i$, we get $\Delta_M(R_{<i})+\Delta_M(p_i) \leq \Delta_M(R)$.  Hence $\Delta_M(p_i) \leq \Delta_M(R-R_{<i})$, so $p_i \precsim R-R_{<i}$, and we choose $R_i \leq R-R_{<i}$ with $R_i \sim p_i$.
\end{proof}

\begin{lem}\label{lem:finitesupply}
    Let $(x_i)_{i \in I}\subseteq M_+$ and put $E = \bigvee_i \supp(x_i)$.  If
    \[
        \sum_{i \in F}\Delta_M(\supp x_i) \leq \Delta_M(1-E)
        \qquad(F  \fin I),
    \]
    then each $x_i$ lies in $EME$, and $(x_i)$ has a projection-supply consisting of projections $R_i \leq 1-E$ with $R_i \sim \supp(x_i)$.  In particular, the hypothesis holds whenever
    \[
        \sum_{i \in F}\Delta_M(\supp x_i) \leq  \tfrac{1}{2}\,1
        \qquad(F  \fin I).
    \]
\end{lem}

\begin{proof}
    The first assertion is Lemma~\ref{lem:packing} with $R = 1-E$.  For the last assertion, Lemma~\ref{lem:subadditive} gives $\Delta_M(E) \leq \tfrac{1}{2}\,1$, hence $\Delta_M(1-E) \geq \tfrac{1}{2}\,1$.
\end{proof}

\begin{proof}[Proof of Theorem~\ref{mainthm:C}]
    Put $p_i = \supp(x_i)$ and $C = \sup_F\sum_{i \in F}\Delta_M(p_i) \in Z(M)_+$, and rescale so that the finite subsums of $(x_i)$ have norm at most $ \tfrac{1}{2}$.  Choose an integer $N>\lVert C\rVert$ and put $B = M_{N+1}(M)$ and $E = e_{00}$, so that $EBE\cong M$.  Identify $Z(B)$ with $Z(M)$.  The $N+1$ diagonal matrix units are equivalent with sum $1$, so $\Delta_B(E) = \frac{1}{N+1}$ and $\Delta_B(1-E) = \frac{N}{N+1}$, and the normalized dimension on the corner $EBE$ is $(N+1)\Delta_B$.  Hence $\Delta_B(p_i) = \frac{1}{N+1}\Delta_M(p_i)$, and for every finite $F$
    \[
        \sum_{i \in F}\Delta_B(p_i) \leq \frac{\lVert C\rVert}{N+1} \leq \frac{N}{N+1} = \Delta_B(1-E).
    \]
    Lemma~\ref{lem:packing} gives pairwise orthogonal projections $R_i \leq 1-E$ with $R_i \sim p_i$, which form a projection-supply for $(x_i)$ in $B$.  The algebra $B$ is finite, hence normal by~\cite{Wri80}, and Theorem~\ref{mainthm:A} in $B$ gives the supremum of the finite subsums, lying in $EBE$.  Under $EBE\cong M$ this is the required supremum in $M_{\sa}$.
\end{proof}

\begin{rem}\label{rem:haagerup}
    When $M$ is a finite factor, Theorem~\ref{mainthm:C} also follows from Haagerup's completeness theorem for the quasitracial metric~\cite[Proposition 3.10]{Haa14}.  The proof above works directly over arbitrary centers and, unlike the metric argument, produces the projection supplies used in Section~\ref{sec:homomorphisms}.  If $M$ admits a center-valued trace, then $M$ is monotone complete by~\cite[Lemma~4]{Ber83}.
\end{rem}

\section{Cardinal monotone completeness and state selection}\label{sec:cardinal}

This section proves results on cardinal completeness and states.  Proposition~\ref{prop:SWcomplete} and Corollary~\ref{cor:SWhomogeneous} are due to Sait\^{o} and Wright~\cite{SW91}.  We include them here for completeness and because the telescoping identity of Lemma~\ref{lem:telescoping} is applied again in the codomain algebra in Section~\ref{sec:homomorphisms}.

\begin{defn}\label{defn:kappacomplete}
    Let $\kappa$ be an infinite cardinal and let $A$ be a unital \cstar-algebra.  Following~\cite[Section 2]{SW91}, $A$ is \textbf{$\kappa$-complete} if the finite subsums of every family of positive elements of $A$, indexed by a set of cardinality at most $\kappa$ and with finite subsums bounded above, have a supremum.  The algebra $A$ is \textbf{$\kappa$-monotone complete} if every norm-bounded increasing net in $A_{\sa}$ indexed by a directed set of cardinality at most $\kappa$ has a supremum.
\end{defn}

\begin{lem}\label{lem:telescoping}
    Let $\mu$ be a limit ordinal and let $(a_\xi)_{\xi < \mu}$ be an increasing family in $A_{\sa}$ such that $a_\lambda = \sup_{\eta < \lambda}a_\eta$ for every limit ordinal $\lambda < \mu$.  Put $d_\xi = a_{\xi+1}-a_\xi$.  Then
    \[
        a_\xi-a_0 = \sup\Bigl\{\sum_{\eta \in F}d_\eta:F \fin \xi\Bigr\}\qquad(\xi < \mu).
    \]
    If the finite subsums of $(d_\xi)_{\xi < \mu}$ have a supremum $s$, then $a_0 + s$ is the supremum of $(a_\xi)_{\xi < \mu}$.
\end{lem}

\begin{proof}
    We argue by transfinite induction on $\xi$; the case $\xi = 0$ is trivial.  At a successor $\xi+1$, the finite subsets of $\xi+1$ containing $\xi$ are cofinal, so~\eqref{eq:translate} and the induction hypothesis give $(a_\xi-a_0)+d_\xi = a_{\xi+1}-a_0$.  At a limit $\lambda$, the finite subsums indexed by $F \fin \lambda$ form the union over $\eta < \lambda$ of those indexed by $F \fin \eta$, so by the induction hypothesis their supremum exists and equals $\sup_{\eta < \lambda}(a_\eta-a_0) = a_\lambda-a_0$.
    
    For the last assertion, $a_0+s$ dominates every $a_\xi$ by the first part.  If $b$ dominates every $a_\xi$ and $F  \fin \mu$, then $F\subseteq\xi$ for some $\xi < \mu$, so $\sum_{\eta \in F}d_\eta \leq a_\xi-a_0 \leq b-a_0$; hence $s\le b-a_0$.
\end{proof}

\begin{lem}\label{lem:directed}
    Suppose every nonempty norm-bounded chain in $A_{\sa}$ of cardinality at most $\kappa$ has a supremum.  Then every nonempty norm-bounded directed subset of $A_{\sa}$ of cardinality at most $\kappa$ has a supremum.
\end{lem}

\begin{proof}
    We argue by induction on the cardinality of the directed set $D$; a finite nonempty directed set has a largest element.  If $D$ is infinite, then by~\cite[Theorem 1]{Mar76} it is the union of an increasing family $(D_\alpha)_{\alpha < |D|}$ of nonempty directed subsets of smaller cardinality.  By induction $s_\alpha = \sup D_\alpha$ exists, the $s_\alpha$ form a norm-bounded chain of cardinality at most $\kappa$, and $\sup_\alpha s_\alpha$ is the supremum of $D$.
\end{proof}

\begin{prop}[{\cite[Section 3]{SW91}}]\label{prop:SWcomplete}
    A unital \cstar-algebra is $\kappa$-monotone complete if and only if it is $\kappa$-complete.
\end{prop}

\begin{proof}
    If $A$ is $\kappa$-monotone complete, the finite subsums of a family as in Definition~\ref{defn:kappacomplete} form a bounded directed set of cardinality at most $\kappa$, so $A$ is $\kappa$-complete.

    Conversely, let $A$ be $\kappa$-complete.  A norm-bounded chain either has a largest element or has a cofinal well-ordered subchain whose order type is an infinite regular cardinal~\cite{Mar76}, and a cofinal subchain has the same upper bounds.  Suppose some chain of cardinality at most $\kappa$ has no supremum, and let $\mu \leq \kappa$ be the least infinite regular cardinal for which some norm-bounded increasing family $(a_\xi)_{\xi < \mu}$ has no supremum.  For a limit $\lambda < \mu$, the family $(a_\eta)_{\eta < \lambda}$ has a cofinal subchain of regular order type $\operatorname{cf}(\lambda) < \mu$, hence a supremum $\tilde a_\lambda$; put $\tilde a_\xi = a_\xi$ at $0$ and at successors.  Then $\tilde a_\xi \leq a_\xi \leq \tilde a_{\xi+1}$, so $(\tilde a_\xi)$ has the same upper bounds as $(a_\xi)$, and it satisfies the continuity hypothesis of Lemma~\ref{lem:telescoping}.  For $F = \{\xi_1 < \dots < \xi_n\}$ we have $\sum_{\xi \in F}(\tilde a_{\xi+1}-\tilde a_\xi) \leq \tilde a_{\xi_n+1}-\tilde a_{\xi_1}$, so $\kappa$-completeness gives a supremum of the finite subsums of the increments, and Lemma~\ref{lem:telescoping} gives a supremum of the chain.  This contradiction shows that all norm-bounded chains of cardinality at most $\kappa$ have suprema, and Lemma~\ref{lem:directed} finishes the proof, since the supremum of the range of a net is the supremum of the net.
\end{proof}

\begin{defn}\label{defn:homogeneous}
    An \AWstar-algebra $M$ is \textbf{$\kappa$-homogeneous} if there are pairwise orthogonal projections $(e_\alpha)_{\alpha < \kappa}$ with $\bigvee_\alpha e_\alpha = 1$ and $e_\alpha \sim 1$ for every $\alpha$.
\end{defn}

Choosing partial isometries $v_\alpha$ with $v_\alpha^*v_\alpha = e_0$ and $v_\alpha v_\alpha^* = e_\alpha$, the elements $e_{\alpha\beta} = v_\alpha v_\beta^*$ form a $\kappa$-indexed system of matrix units with $\sum_\alpha e_{\alpha\alpha} = 1$ and $e_{00} \sim 1$; conversely such a system gives $\kappa$-homogeneity.  A $\kappa$-homogeneous \AWstar-algebra is properly infinite, and every properly infinite \AWstar-algebra is $\aleph_0$-homogeneous~\cite[Section 17, Theorem 1]{Ber72}.

\begin{cor}[{\cite[Theorem 2.5]{SW91}}]\label{cor:SWhomogeneous}
    Every $\kappa$-homogeneous \AWstar-algebra is $\kappa$-monotone complete.
\end{cor}

\begin{proof}
    By~\cite[Theorem 2.5]{SW91} such an algebra is $\kappa$-complete.  Now apply Proposition~\ref{prop:SWcomplete}.
\end{proof}

Corollary~\ref{cor:SWhomogeneous} can also be recovered from Corollary~\ref{cor:copysupply}, applied to $E = e_0$ with the copy supply $(e_\alpha)_{\alpha \geq 1}$, together with Proposition~\ref{prop:SWcomplete}.  That route uses normality from~\cite{AH26}, which in turn rests on~\cite{SW91}, so it is not an independent proof.

Sait\^{o} and Wright~\cite[Theorem~2.1.14]{SW15} prove the $\kappa = \aleph_0$ case of the following proposition.  It is phrased there with a single faithful state, but a countable separating family of states can be turned into a single faithful state via a weighted sum.  In the version below, $\kappa$ separating states select a directed subset of cardinality at most $\kappa$ with the same supremum.

\begin{prop}\label{prop:stateselection}
    Let $\kappa$ be an infinite cardinal, and let $A$ be a $\kappa$-monotone complete unital \cstar-algebra admitting ordinary states $(\varphi_\alpha)_{\alpha < \kappa}$ that separate $A_+\setminus\{0\}$ from $0$.  Then every norm-bounded increasing net $(a_i)_{i \in I}$ in $A_{\sa}$ contains a directed subset $J\subseteq I$ with $|J| \leq \kappa$ and
    \[
        \sup_{j \in J}a_j = \sup_{i \in I}a_i .
    \]
    In particular $A$ is monotone complete.
\end{prop}

\begin{proof}
After translation and rescaling, $0 \leq a_i \leq 1$.  For every nonempty directed $K\subseteq I$ with $|K| \leq \kappa$, the supremum $a_K = \sup_{k \in K}a_k$ exists.  For $\alpha < \kappa$ put
\[
    c_\alpha = \sup\{\varphi_\alpha(a_K):K\subseteq I\text{ nonempty directed},\ |K| \leq \kappa\},
\]
and for each $n \geq 1$ choose such a $K_{\alpha,n}$ with $\varphi_\alpha(a_{K_{\alpha,n}})>c_\alpha-1/n$.  The union of the $K_{\alpha,n}$ has cardinality at most $\kappa$, and closing it under chosen upper bounds for pairs through countably many stages gives a directed $J\subseteq I$ with $|J| \leq \kappa$.  Put $a = a_J$.  Then $\varphi_\alpha(a) = c_\alpha$ for every $\alpha$: on one hand $a \geq  a_{K_{\alpha,n}}$ for every $n$, and on the other $J$ occurs in the definition of $c_\alpha$.

Fix $i \in I$, enlarge $J\cup\{i\}$ to a directed $K\subseteq I$ with $|K| \leq \kappa$, and put $b = a_K$.  Then $b \geq a$ and $b \geq a_i$, while $\varphi_\alpha(b) \leq c_\alpha = \varphi_\alpha(a)$ for every $\alpha$.  Since $b-a \geq 0$ and the states separate positive elements from zero, $b = a$, so $a_i \leq a$.  Any upper bound of the net dominates $(a_j)_{j \in J}$, and hence $a$.  Therefore $a = \sup_i a_i$.
\end{proof}

We next show that the hypotheses of Proposition~\ref{prop:stateselection} hold in every properly infinite \AWstar-factor with a corner admitting a faithful state.

\begin{lem}\label{lem:absorption}
    Let $q = q_1+q_2$ be projections in $M$ with $q_1\sim q_2\sim q$, and let $r$ be a projection orthogonal to $q$ with $r\precsim q$.  Then $q+r\sim q$.
\end{lem}

\begin{proof}
    Choose $r' \leq q_2$ with $r'\sim r$.  Since $q\sim q_1$, $r\sim r'$, $q\perp r$ and $q_1\perp r'$, we get $q+r\sim q_1+r' \leq q$.  Together with $q \leq q+r$, the Schröder–Bernstein theorem for projections~\cite{Kap51} gives $q+r\sim q$.
\end{proof}

\begin{lem}\label{lem:faithfulcorner}
    Let $p,r \in \Proj(M)$ with $r\ne0$ and $r\precsim p$.  If $pMp$ admits a faithful state, then so does $rMr$.
\end{lem}

\begin{proof}
    Let $\psi$ be a faithful state on $pMp$, choose $r' \leq p$ with $r'\sim r$ and a partial isometry $v$ with $v^*v = r$ and $vv^* = r'$.  Then $\psi(r')>0$, and $\psi_r(x) = \psi(r')^{-1}\psi(vxv^*)$ is a state on $rMr$.  If $x \in (rMr)_+$ and $\psi_r(x) = 0$, then $vxv^* = 0$, so $x = v^*(vxv^*)v = 0$.
\end{proof}

\begin{prop}\label{prop:faithfulcorner}
    Let $M$ be an \AWstar-factor, and suppose there is a nonzero projection $p \in  M$ such that $pMp$ admits a faithful ordinary state.  Then $M$ is monotone complete.  If $M$ is properly infinite, then for some infinite cardinal $\kappa$ the algebra $M$ is $\kappa$-homogeneous and admits $\kappa$ ordinary states separating $M_+\setminus\{0\}$ from $0$.
\end{prop}

\begin{proof}
    If $M$ is finite of type $\mathrm{II}_1$, then $pMp$ is a type $\mathrm{II}_1$ \AWstar-factor with a faithful ordinary state, hence a \wstar-algebra by~\cite[Corollary 7]{Wri76}.  Every nonzero projection of a factor has central cover $1$, so $M$ is a \wstar-algebra by~\cite[Corollary 1]{EGW83}.  The finite type $\mathrm{I}$ case is classical.
    
    Let $M$ be properly infinite.  By~\cite[Section 17, Theorem 1]{Ber72} there are pairwise orthogonal projections $(e_n)_{n < \omega}$ with $e_n \sim 1$.  Since $p\precsim e_n$, choose $p_n \leq e_n$ with $p_n\sim p$, and extend $(p_n)$ by Zorn's lemma to a maximal orthogonal family $(p_i)_{i \in I}$ of projections equivalent to $p$.  Put $\kappa = |I| \geq \aleph_0$, $q = \bigvee_i p_i$ and $r = 1-q$.  By maximality $p\not\precsim r$, so generalized comparability~\cite[Section 14, Corollary 1]{Ber72}, specialized to the factor $M$, gives $r\precsim p$.
    
    Partition $I$ into sets $(I_\alpha)_{\alpha < \kappa}$ of cardinality $\kappa$ and put $Q_\alpha = \bigvee_{i \in I_\alpha}p_i$.  Since all $p_i$ are equivalent, Fact~\ref{fact:addability} and bijections between index sets of cardinality $\kappa$ show that $Q_\alpha\sim q$, that $q-Q_0 = \bigvee_{\alpha \geq 1}Q_\alpha\sim q$, and that $Q_0$ splits as a sum of two projections each equivalent to $Q_0$.  As $r\precsim p\precsim Q_0 \leq q$, Lemma~\ref{lem:absorption} gives $1 = q+r\sim q$ and $Q_0+r\sim Q_0\sim q \sim 1$.  Thus $Q_0+r$ and the $Q_\alpha$ with $\alpha \geq 1$ are $\kappa$ pairwise orthogonal projections, each equivalent to $1$, with join $1$, and $M$ is $\kappa$-homogeneous.
    
    By Lemma~\ref{lem:faithfulcorner}, each $p_i M p_i$ admits a faithful state $\psi_i$, and if $r\ne0$ then $rMr$ admits a faithful state $\psi_r$.  Put $\varphi_i (x) = \psi_i (p_i x p_i)$ and, when $r\ne0$, $\varphi_r(x) = \psi_r(rxr)$.  If $x \geq 0$ and all these states vanish at $x$, then $x^{1/2}p_i = 0$ for every $i$ and $x^{1/2}r = 0$; since $1 = r\vee\bigvee_i p_i$, Lemma~\ref{lem:annihilator} gives $x = 0$.  Repeating states if necessary, this is a $\kappa$-indexed separating family.  Monotone completeness now follows from Corollary~\ref{cor:SWhomogeneous} and Proposition~\ref{prop:stateselection}.
\end{proof}

\section{Join-preserving homomorphisms}\label{sec:homomorphisms}

Join-preserving $*$-homomorphisms are the usual morphisms in the theory of \AWstar-algebras~\cite{Wid56} (see~\cite{HL,HR14} for categorical treatments).  By a lemma of Feldman and Fell~\cite[Lemma~2]{FF57} (see also~\cite[Lemma~2.18]{Gow}), if a $*$-homomorphism preserves the join of an orthogonal family $(p_i)$, then it preserves the join of every orthogonal family $(q_i)$ with $q_i \sim p_i$ for each $i$.

\begin{defn}\label{defn:joinpreserving}
    Let $\kappa$ be a cardinal.  A $*$-homomorphism $\pi:M\to N$ between \AWstar-algebras is \textbf{$\kappa$-join-preserving} if $\pi(\bigvee_i p_i) = \bigvee_i \pi(p_i)$ for every orthogonal family $(p_i)$ of at most $\kappa$ projections in $M$, and \textbf{join-preserving} if this holds for every orthogonal family.  A \cstar-subalgebra $M$ of an \AWstar-algebra $N$ is an \textbf{\AWstar-subalgebra} if $M$ is an \AWstar-algebra and the inclusion is join-preserving.
\end{defn}

\begin{proof}[Proof of Theorem~\ref{mainthm:B}]
    After rescaling we may assume $\sum_{i \in F}x_i \leq  \tfrac{1}{2}E$.  The algebras $M$ and $N$ are normal by~\cite[Theorem A]{AH26}.  Theorem~\ref{mainthm:A} gives pairwise orthogonal projections $(q_i)$ with $E q_i E = x_i$, such that $x = EqE$ is the supremum of the finite subsums, where $q = \bigvee_i q_i$.  Since $\pi$ is $|I|$-join-preserving, $\pi(q) = \bigvee_i \pi(q_i)$, and the $\pi(q_i)$ are pairwise orthogonal.  Proposition~\ref{prop:joinsup} in $N$, with the projection $\pi(E)$, gives
    \[
    \pi(x) = \pi(E)\pi(q)\pi(E) = \sup_{N,\;F  \fin I}\ \sum_{i \in F}\pi(E)\pi(q_i)\pi(E) = \sup_{N,\;F  \fin I}\ \sum_{i \in F}\pi(x_i). \qedhere
    \]
\end{proof}

Theorem~\ref{mainthm:B} is of interest when $N$ is an arbitrary \AWstar-algebra, and because it requires only the joins of families of cardinality $|I|$ to be preserved.

\begin{rem}\label{rem:encoding}
    The proof uses the dilation of Theorem~\ref{mainthm:A} as an encoding: the order sum of the noncommuting elements $x_i$ is written as the compression of the join of the orthogonal projections $q_i$.  A $*$-homomorphism carries compressions to compressions, and an $|I|$-join-preserving one carries this particular join to the corresponding join in $N$, where normality converts it back into an order sum.  Only the joins of families of cardinality $|I|$ enter the hypothesis.
\end{rem}

\begin{cor}\label{cor:homogeneous}
    Let $\kappa$ be an infinite cardinal, let $M$ be a $\kappa$-homogeneous \AWstar-algebra, let $N$ be an \AWstar-algebra, and let $\pi:M\to N$ be a $\kappa$-join-preserving $*$-homomorphism.  Then $\pi(\sup_ja_j) = \sup_{N}\pi(a_j)$ for every norm-bounded increasing net $(a_j)_{j \in J}$ in $M_{\sa}$ with $|J| \leq \kappa$.  In particular, a $*$-homomorphism from a properly infinite \AWstar-algebra into an \AWstar-algebra that preserves joins of countable orthogonal families preserves suprema of bounded increasing sequences.
\end{cor}

\begin{proof}
    The suprema in $M$ exist by Corollary~\ref{cor:SWhomogeneous}.
    
    Let $(x_i)_{i \in I}$ be positive with finite subsums bounded above, where $|I| \leq \kappa$.  Let $(e_\alpha)_{\alpha < \kappa}$ be as in Definition~\ref{defn:homogeneous}, put $E = e_0$, and choose a partial isometry $v$ with $v^*v = 1$ and $vv^* = E$.  Put $x_i' = v x_i v^* \in EME$ and choose an injection $\iota:I\to\kappa\setminus\{0\}$.  The projections $e_{\iota(i)}$ form a projection-supply for $(x'_i)$, since $\supp(x'_i) \leq E \sim e_{\iota(i)}$.  By Theorem~\ref{mainthm:B}, $\pi(\sup_Fx'_F) = \sup_N\pi(x'_F)$, and this supremum lies in $\pi(E)N\pi(E)$.  Now $\operatorname{Ad}v^*$ is an order isomorphism of $EME$ onto $M$ carrying $x'_F$ to $x_F$, and $y\mapsto\pi(v)^*y\pi(v)$ is an order isomorphism of $\pi(E)N\pi(E)$ onto $\pi(1)N\pi(1)$ carrying $\pi(x'_F)$ to $\pi(x_F)$.  Together with Lemma~\ref{lem:corner}, this gives $\pi(\sup_Fx_F) = \sup_N\pi(x_F)$.
    
    A norm-bounded chain either has a largest element or has a cofinal well-ordered subchain whose order type is an infinite regular cardinal~\cite{Mar76}.  A cofinal subchain has the same upper bounds, and so does its image under $\pi$, so it suffices to treat families $(a_\xi)_{\xi < \mu}$ with $\mu \leq \kappa$ regular, by induction on $\mu$.  Replace $a_\lambda$, for limit $\lambda < \mu$, by $\tilde a_\lambda = \sup_{\eta < \lambda}a_\eta$; the resulting family $(\tilde a_\xi)$ is mutually cofinal with the original one, and by the induction hypothesis $\pi(\tilde a_\lambda) = \sup_N\{\pi(a_\eta):\eta < \lambda\}$.  Put $d_\xi = \tilde a_{\xi+1}-\tilde a_\xi$ and $b_\xi = \pi(\tilde a_\xi)$.  Since $\tilde a_\eta \leq a_\eta \leq \tilde a_{\eta+1}$, the families $(b_\eta)_{\eta < \lambda}$ and $(\pi(a_\eta))_{\eta < \lambda}$ are mutually cofinal at every limit $\lambda$, and hence $b_\lambda = \sup_{\eta < \lambda}b_\eta$.  By the first step the finite subsums of $(d_\xi)$ have a supremum $s$ with $\pi(s) = \sup_N\sum_{\xi \in F}\pi(d_\xi)$.  Lemma~\ref{lem:telescoping} in $M$ shows that $\tilde a_0+s$ is the supremum of $(a_\xi)$.  The family $(b_\xi)$ is therefore increasing and continuous at limits, with increments $\pi(d_\xi)$, so Lemma~\ref{lem:telescoping} in $N$ shows that $b_0+\pi(s)$ is the supremum of $(b_\xi)$, which is also the supremum of $(\pi(a_\xi))$.  Hence $\pi(\sup_\xi a_\xi) = \sup_N\pi(a_\xi)$.
    
    We argue by induction on the cardinality of a norm-bounded directed set $D\subseteq M_{\sa}$ with $|D| \leq \kappa$.  By~\cite[Theorem 1]{Mar76}, $D$ is the union of an increasing family $(D_\alpha)_{\alpha < |D|}$ of directed subsets of smaller cardinality.  By induction, $\pi(\sup D_\alpha) = \sup_N\pi(D_\alpha)$, and $(\sup D_\alpha)_\alpha$ is a chain whose supremum is $\sup D$.  The previous step gives $\pi(\sup D) = \sup_N\pi(\sup D_\alpha) = \sup_N\pi(D)$.  Applying this to the range of a net proves the first assertion.  The last assertion is the case $\kappa = \aleph_0$.
\end{proof}

\begin{cor}\label{cor:normalhom}
    Let $\kappa$ be an infinite cardinal.  Suppose $M$ is a $\kappa$-homogeneous \AWstar-algebra admitting $\kappa$ ordinary states that separate $M_+\setminus\{0\}$ from $0$.  Then $M$ is monotone complete, and every $\kappa$-join-preserving $*$-homomorphism $\pi$ from $M$ into an \AWstar-algebra $N$ satisfies $\pi(\sup_i a_i) = \sup_N\pi(a_i)$ for every norm-bounded increasing net $(a_i)$ in $M_{\sa}$.
\end{cor}

\begin{proof}
    By Corollary~\ref{cor:SWhomogeneous} and Proposition~\ref{prop:stateselection}, $M$ is monotone complete, and a norm-bounded increasing net $(a_i)_{i \in I}$ contains a directed subset $J\subseteq I$ with $|J| \leq \kappa$ and $\sup_{j \in J}a_j = \sup_{i \in I}a_i = a$.  By Corollary~\ref{cor:homogeneous}, $\pi(a) = \sup_N\{\pi(a_j):j \in J\}$.  Since $\pi(a)$ dominates every $\pi(a_i)$, and every upper bound of $(\pi(a_i))_{i \in I}$ dominates $(\pi(a_j))_{j \in J}$, we get $\pi(a) = \sup_N\pi(a_i)$.
\end{proof}

Under the hypotheses of Corollary~\ref{cor:normalhom}, every orthogonal family of nonzero projections in $M$ has cardinality at most $\kappa$, since each state is positive on at most countably many members of such a family.  Hence $\kappa$-join preservation is equivalent to join preservation there, and when $N$ is a \wstar-algebra the conclusion also follows from~\cite[9.34]{SZ}.  The content of Corollary~\ref{cor:normalhom} lies in allowing $N$ to be an arbitrary \AWstar-algebra.

\begin{rem}\label{rem:sharp}
    The cardinal in Corollary~\ref{cor:normalhom} cannot be lowered if suitable large cardinals exist.  For an infinite cardinal $\kappa$, the algebra $B(\ell^2(\kappa))$ is $\kappa$-homogeneous, and the vector states of an orthonormal basis are $\kappa$ states separating $B(\ell^2(\kappa))_+\setminus\{0\}$ from $0$.  Blecher and Weaver~\cite{BW} show that $B(\ell^2(\kappa))$ admits a singular pure state that is additive on countable orthogonal families of projections if and only if $\kappa$ is Ulam measurable, and one that is additive on orthogonal families of fewer than $\kappa$ projections if and only if $\kappa$ is measurable.  Let $\varphi$ be such a state, with GNS representation $(\pi_\varphi,H_\varphi,\xi_\varphi)$, and let $(p_n)$ be an orthogonal family of projections of the relevant size, with join $p$.  For every unitary $u$, the projections $u^*p_nu$ are orthogonal with join $u^*pu$, so
    \[
    \Bigl\langle\Bigl(\pi_\varphi(p)-\bigvee_n\pi_\varphi(p_n)\Bigr)\pi_\varphi(u)\xi_\varphi,\ \pi_\varphi(u)\xi_\varphi\Bigr\rangle=\varphi(u^* p u)-\sum_n\varphi(u^* p_n u)=0 .
    \]
    The operator in the brackets is positive, and the vectors $\pi_\varphi(u)\xi_\varphi$ span a dense subspace of $H_\varphi$, so it vanishes.  Thus $\pi_\varphi$ preserves the joins of these families, but $\pi_\varphi$ is not normal since $\varphi$ is singular.  Consequently, if $\kappa$ is Ulam measurable, preservation of joins of countable orthogonal families does not suffice in Corollary~\ref{cor:normalhom}, and if $\kappa$ is measurable, neither does preservation of joins of orthogonal families of fewer than $\kappa$ projections.
\end{rem}

\begin{cor}\label{cor:subalgebra}
    Let $M$ be an \AWstar-subalgebra of an \AWstar-algebra $N$.  If $M$ is $\kappa$-homogeneous, then every bounded increasing net in $M_{\sa}$ of cardinality at most $\kappa$ has the same supremum in $M$ and in $N$.  If in addition $M$ admits $\kappa$ separating ordinary states, then $M$ is monotone complete and every supremum in $M_{\sa}$ of a bounded increasing net is also its supremum in $N_{\sa}$.  In particular, for every properly infinite \AWstar-subalgebra $M$ of $N$, suprema of bounded increasing sequences in $M$ are suprema in $N$.
\end{cor}

\begin{proof}
    Apply Corollaries~\ref{cor:homogeneous} and~\ref{cor:normalhom} to the inclusion.
\end{proof}

\begin{cor}\label{cor:finitesource}
    Let $M$ be a finite \AWstar-algebra and let $(x_i)_{i \in I}\subseteq M_+$ have uniformly norm-bounded finite subsums.  Put $E = \bigvee_i \supp(x_i)$ and suppose
    \[
        \sum_{i \in F}\Delta_M(\supp x_i) \leq \Delta_M(1-E) \qquad(F \fin I).
    \]
    Then every $|I|$-join-preserving $*$-homomorphism $\pi$ from $M$ into an \AWstar-algebra $N$ satisfies
    \[
        \pi\Bigl(\sup_F\sum_{i \in F}x_i\Bigr) = \sup_{N,\;F}\ \sum_{i \in F}\pi(x_i).
    \]
    In particular, the conclusion holds if $\sum_{i \in F}\Delta_M(\supp x_i) \leq  \tfrac{1}{2}\,1$ for every finite $F$.
\end{cor}

\begin{proof}
    By Lemma~\ref{lem:finitesupply}, each $x_i$ lies in $EME$ and $(x_i)$ has a projection-supply in $1-E$.  Apply Theorem~\ref{mainthm:B}.
\end{proof}

\begin{rem}\label{rem:homknown}
    When $N$ is the regular monotone completion of a normal \AWstar-algebra $M$, the inclusion preserves all existing suprema by~\cite{Ham81}, and in the von Neumann setting countably additive $*$-homomorphisms were studied in~\cite{BuH95}.  
    For a type I \AWstar-algebra $B$, Hamana asks whether an \AWstar-subalgebra of $B$ is automatically monotone closed in $B$, and notes that a positive answer in general would make every normal \AWstar-algebra monotone complete~\cite[Remark 4(i)]{Ham83}.  Corollary~\ref{cor:subalgebra} gives a domain-side criterion for monotone closedness in arbitrary \AWstar-algebras, upgrading the previously established monotone completeness by showing the suprema computed in $M$ are the suprema of the ambient algebra.  In Corollary~\ref{cor:finitesource} the codomain need not be finite, and the projection-supply lies in $1-E$ and is not required to contain copies of $E$, so neither a finite-codomain metric argument nor the matrix-unit mechanism of~\cite{SW91} gives the stated conclusion.  When $N$ is a \wstar-algebra the result also follows from~\cite[9.33]{SZ}, so Corollary~\ref{cor:subalgebra} is of interest when $N$ is an arbitrary \AWstar-algebra.
\end{rem}

See also Sait\^{o}'s work on \AWstar-algebras with the monotone convergence property~\cite{Sai79}.

\end{document}